\documentclass[11pt,twoside]{amsart}
\usepackage[latin1]{inputenc}
\usepackage[T1]{fontenc}
\usepackage{mathtools}
\usepackage{graphicx}
\usepackage{tikz}
\usetikzlibrary{chains}
\usepackage{tcolorbox}

\PassOptionsToPackage{pdfusetitle,pagebackref,colorlinks}{hyperref}
\usepackage{bookmark}
\hypersetup{
  linkcolor={red!70!black},
  citecolor={green!70!black},
  urlcolor={blue!80!black}
}

\usepackage[latin1]{inputenc}
\usepackage{amsmath}
\usepackage{amsthm}
\usepackage{amssymb}

\usepackage[all]{xy}
\usepackage{hyperref}
\newtheorem{thm}{Theorem}[section]

\newtheorem{prop}[thm]{Proposition}

\newtheorem{lem}[thm]{Lemma}
\newtheorem{cor}[thm]{Corollary}
\newtheorem{conj}[thm]{Conjecture}

\numberwithin{equation}{section}

\theoremstyle{definition}

\newtheorem{remark}[thm]{Remark}
\newtheorem{ex}[thm]{Example}
\usepackage[OT2,T1]{fontenc}
\DeclareSymbolFont{cyrletters}{OT2}{wncyr}{m}{n}
\DeclareMathSymbol{\Sha}{\mathalpha}{cyrletters}{"58}

\newcommand{\Aut}{{\rm Aut}}

\newcommand{\Pic}{{\rm Pic}}

\newcommand{\Hom}{{\rm Hom}}

\newcommand{\coker}{{\rm coker}}
\newcommand{\cal}{\mathcal}

\newcommand{\kc}{{\cal C}}

\newcommand{\ki}{{\cal I}}
\newcommand{\kk}{{\cal K}}

\newcommand{\ko}{{\cal O}}

\newcommand{\kt}{{\cal T}}

\newcommand{\ZZ}{\mathbb{Z}}
\newcommand{\QQ}{\mathbb{Q}}

\newcommand{\CC}{\mathbb{C}}

\newcommand{\PP}{\mathbb{P}}

\renewcommand{\to}{\xymatrix@1@=15pt{\ar[r]&}}
\renewcommand{\rightarrow}{\xymatrix@1@=15pt{\ar[r]&}}
\renewcommand{\leftarrow}{\xymatrix@1@=15pt{&\ar[l]}}
\renewcommand{\mapsto}{\xymatrix@1@=15pt{\ar@{|->}[r]&}}
\renewcommand{\twoheadrightarrow}{\xymatrix@1@=18pt{\ar@{->>}[r]&}}
\renewcommand{\hookrightarrow}{\xymatrix@1@=15pt{\ar@{^(->}[r]&}}
\newcommand{\hook}{\xymatrix@1@=15pt{\ar@{^(->}[r]&}}
\newcommand{\congpf}{\xymatrix@L=0.6ex@1@=15pt{\ar[r]^-\sim&}}
\renewcommand{\cong}{\simeq}

\begin{document}

\title[]{Maximal variation of curves on K3 surfaces}
\author[Y.\ Dutta \& D.\ Huybrechts]{Y. Dutta \& D.\ Huybrechts}

\address{Mathematisches Institut and Hausdorff Center for Mathematics,
Universit{\"a}t Bonn, Endenicher Allee 60, 53115 Bonn, Germany}
\email{ydutta@uni-bonn.de, huybrech@math.uni-bonn.de}

\begin{abstract} \noindent
We prove that curves in a non-primitive, base point free, ample linear system on a K3 surface have maximal variation. The result is deduced from general restriction theorems applied to the tangent bundle. We also show how to use specialisation to spectral curves to deduce information about the variation of curves contained in a K3 surface more directly. The situation for primitive linear systems is not clear at the moment.
However, the maximal variation holds in genus two and can, in many cases,
be deduced  from a recent result of van Geemen and Voisin \cite{VGV} confirming a conjecture
due to Matsushita. 

 \vspace{-2mm}
\end{abstract}

\maketitle
{\let\thefootnote\relax\footnotetext{The first author is funded
by the Hausdorff Center for Mathematics, Bonn (Germany's Excellence Strategy--EXC-2047/1--390685813). The second author is supported by the 
ERC Synergy Grant HyperK.}
\marginpar{}
}
\section{Introduction}
Smooth curves contained in K3 surfaces are special yet general. This has been
a guiding principle for important work over the last decades, cf.\ \cite{BeauvBourb,Laz,VoisinBourbaki}. For dimension reasons,
 a general  curve of high genus cannot be contained in any
smooth K3 surface, although those that are 
 behave in many respects like
 a general curve. It is a basic question to ask in how many ways, if at all, a curve
 can be embedded into a K3 surface. In other words,  how much do curves
 contained in a K3 surface vary, either within the given K3 surface or together with it?
\smallskip

\subsection{} 
Our main result is concerned with the variation of curves in an arbitrary but fixed
K3 surface.

\begin{thm}\label{thm:main}
Let $|H|$ be a base point free, ample linear system on a K3 surface such that the locus of non-reduced curves in $|H|$ has codimension at least three. 
Then, for all $m\geq 2$, the family of smooth curves in the linear system
$|mH|$ has maximal variation. 
\end{thm}

In other words, a generic curve in a non-primitive, base point free ample linear system occurs only finitely many times in its linear system. The result can also be expressed
by saying that the rational
map $\xymatrix{|mH|\ar@{-->}[r]& M_{g_m}}$ that associates with a smooth curve $C\subset S$ its isomorphism class $[C]\in M_{g_m}$ in the moduli space of smooth curves of genus $g_m\coloneqq\dim|mH|$ is generically
quasi-finite as soon as $m\geq2$.  

For $m\geq 3$ the assertion holds without any assumption on the non-reduced locus and we expect this to be true also for $m=2$. For $m=1$, we
can currently prove maximal variation only for $(H.H)=2$. It would
follow in this generality from a conjecture of Matsushita, which has been proved for
generic K3 surfaces by van Geemen and Voisin \cite{VGV}.

\subsection{} The problem of deforming curves in varying K3 surfaces was already studied by Mukai \cite{Mukai} who showed that the generic fibre of
the map $(C\subset S)\mapsto C$  is finite if the genus satisfies $g=11$ or $g\geq 13$. Here, $C$ is a primitive smooth curve 
contained in a varying K3 surface $S$.

Results due to Arbarello-- Bruno--Sernesi \cite{ABS14} and to Feyzbakhsh \cite{Fey} showed that for a polarized K3 surface $(S,H)$ with the same hypotheses on the genus
and with $\Pic(S)=\ZZ\cdot H$ the K3 surface $S$ can be uniquely reconstructed from
any smooth curve $C\in |H|$. In particular, the map $(C\subset S)\mapsto C$ is injective at such a point and then, of course, for fixed $S$ the rational map
$\xymatrix{|H|\ar@{-->}[r]& M_{g}}$ has finite fibres.
Results without restrictions on the Picard number have been proved
by Arbarello-- Bruno--Sernesi \cite{ABS17}, Ciliberto--Dedieu--Sernesi \cite{CDS},
and Ciliberto--Dedieu \cite{CD}. Their results cover most of the cases dealt with
by Theorem \ref{thm:main} using completely different methods.
In fact, they can handle also the case $m=1$ albeit with
the additional assumption on the Clifford index to be at least three. Another advantage of their approach is that they
prove results for all smooth curves in a linear system, while our approach is mostly
restricted to generic curves.


An infinitesimal analogue, proving the injectivity of the tangent map under similar conditions, was worked out by Totaro \cite{Totaro}. These results imply maximal variation
for primitive and therefore, also non-primitive ample linear systems of high enough genus on a generic polarized K3 surfaces.  We also wish to draw attention to the recent paper
\cite{CG} where maximal variation of the normalization of singular curves of unbounded
degree is studied.

\subsection{} In this note, we study curves contained in an arbitrary but fixed K3 surface $S$.
Without any condition on the Picard group of $S$ or on the genus of the curves,
our result says that curves in a non-primitive, base point free, ample linear system have maximal variation, i.e.\ up to finite ambiguity a generic curve $C$ is embedded uniquely in its ambient K3 surface.

Under additional assumptions on $m$ or $(H.H)$ the result follows directly from
existing results.
 For example, if $(H.H)\geq 48$, it is an immediate consequence of the work of Hein \cite{Hein2}. For $m\geq 6$ it follows, at least for non-hyperelliptic curves, from a general restriction result due to Balaji--Koll\'ar \cite{BaKo} and for $m\geq5$ from Totaro's Bott vanishing theorems \cite{Totaro}. Our approach is logically independent of these results and uses as the main input another result of Hein \cite{HeinDiss} which we state and prove in Section \ref{sec:Hein}.

For K3 surfaces with maximal Mumford--Tate group, maximal variation as claimed
here is related to a recent result of van Geemen and Voisin \cite{VGV}. Establishing the result in the non-generic case lends further strong evidence for a conjecture of Matsushita.

\subsection{} The original idea of our approach was to deduce maximal variation for
curves contained in a K3 surface from a similar statement for spectral curves. More precisely, instead of studying curves in a linear system $|mH|$ on a K3 surface
one can look at spectral covers $D\twoheadrightarrow C$ of degree $m$ of
a distinguished primitive curve $C\in |H|$. 

Here, one thinks of the spectral curve $D$ as being
contained in the cotangent bundle  ${\mathbb V}(\omega_C)$ of $C$. So the projective K3 surface $S$
is replaced by the quasi-projective symplectic surface
${\mathbb V}(\omega_C)$  and curves in $|mH|$  specialise to spectral covers ${\mathbb V}(\omega_C)\supset D\twoheadrightarrow C$. This degeneration technique was first studied by Donagi--Ein--Lazarsfeld \cite{DEL}
and exploited in a recent paper of de Cataldo--Maulik--Shen \cite{MaulikShen}.

It turns out that the variation of spectral curves is not maximal, so Theorem \ref{thm:main} fails when K3 surfaces are replaced by ${\mathbb V}(\omega_C)$.
However, due to a result of Hodge and Mulase \cite{HM}, the isotrivial families of
spectral curves are well understood, which allows one to bound the
dimension of isotrivial families in $|mH|$ for $m\geq2$.  Assuming Matsushita's conjecture
 that Lagrangian fibrations of hyperk\"ahler manifolds have
zero or maximal variation, would then immediately yield maximal variation in $|mH|$ for $m\geq2$.
Thus, the verification of the conjecture  \cite{VGV} for K3 surfaces 
with maximal Mumford--Tate group and Picard number $\rho\leq 17$
gives an alternative proof for the maximal variation in these cases.


\subsection{Outline of content} 
In Section \ref{sec:HBK} we recall and prove a result due to Hein, 
see Theorem \ref{thm:Hein}, from which we conclude
that $h^0(\kt_S|_C)\leq 2$ for a generic ample curve $C$ on $S$, see Corollary \ref{cor:Hein}.
Next, we  explain how to use a result of Balaji--Koll\'ar
to prove Theorem \ref{thm:main} for $|H|$ non-hyperelliptic and $m\geq 6$,
see Corollary \ref{cor:BK}.
The main result of Section \ref{sec:HBK} and the key technical result of the paper is
Theorem \ref{thm:main2} which shows $h^0(\kt_S|_C)=0$ for a generic $C\in |mH|$, $m\geq 2$.
The section concludes with Remark \ref{rem:m=3} addressing the case $m\geq 3$ without additional assumptions on the non-reduced locus.

In Section \ref{sec:maxvar} we study the rational map 
$\xymatrix{|H|\ar@{-->}[r]&M_g}$, describe its derivative and
 show that its generic fibre is of dimension at most two, see Corollary \ref{cor:dim2}. Theorem \ref{thm:main} is then an immediate consequence, its proof is presented in Section \ref{sec:proof}. It is here where the technical assumption on the non-reduced locus is used.
In this section, we also provide examples to show that neither ampleness nor base point freeness can be dropped from the hypothesis of our main theorem.
In the rest of this section, we explain the relation to Totaro's Bott vanishing results.

In Section \ref{sec:Hitchin} we study the analogous situation provided by spectral curves and the Hitchin system. After briefly recalling the basic setup we give a short proof of a result due to Hodge and Mulase showing that every spectral curve of degree at least two is contained in a $(g+1)$-dimensional isotrivial family, see Corollary \ref{cor:HM}. We also explain how to use specialisation to the normal cone
 to deduce from the spectral curves situation information about variation
 of curves on K3 surfaces, see Corollary \ref{cor:atmostg}. The result is then combined
with the verification of  Matsushita's conjecture by van Geemen and Voisin for very general
hyperk\"ahler manifolds to give an alternative proof of Theorem \ref{thm:main} for very general K3 surfaces without additional conditions on $|H|$ and without using 
any information about the restriction of the tangent bundle.

The final Section \ref{sec:counterexamples} contains more on the relation to Matsushita's conjecture and the proof of Theorem \ref{thm:main} for $g=2$ and $m\geq1$.

\subsection{Conventions}
 The main result is proved for smooth projective K3 surfaces over an algebraically closed field $k$ of characteristic zero. In Sections \ref{sec:Hitchin} and \ref{sec:counterexamples} we mention the Hodge theoretic result \cite{VGV} for
 which $k=\CC$ is needed. 
 
 By $H$ we denote a base point free, ample line bundle on $S$, not necessarily primitive. 
The linear system $|H|$ is called non-hyperelliptic if it contains a smooth non-hyperelliptic curve. Any very ample linear system is non-hyperelliptic. The degree of $S$ with respect to $H$ is $d\coloneqq (H.H)$ and the genus of $H$ is the genus of a smooth
curve in $|H|$, i.e.\
$2g-2=d$ or, in other words, $g=\dim |H|$.

\noindent
{\bf Acknowledgements.}  We wish to thank Tim B\"ulles and John Ottem for  email correspondences and Richard Thomas for enquiring about the current status of the problem which rekindled our interest. Thanks to Edoardo Sernesi for his help with the literature on singular curves. We are grateful to Thomas Dedieu and Edoardo Sernesi 
for comments on the first version of the article and explanations concerning their results.

\section{Restricting the tangent bundle}\label{sec:HBK}
As we will recall in Section \ref{sec:maxvar}, the infinitesimal nature of the rational map $\xymatrix{|H|\ar@{-->}[r]& M_g}$ at a smooth curve $C\in |H|$
is determined by the restriction $\kt_S|_C$ of the tangent bundle $\kt_S$ of the K3 surface $S$, see Lemma \ref{lem:tangent}. Ideally
one would like $\kt_S|_C$ to be stable, at least for the generic
curve $C\in |H|$, since stability ensures the vanishing of global sections of $\kt_S|_C$ which is crucial for our main theorem.  However, semistability can also be exploited at the expense of significantly more work. This is the core of our main technical result.
 All three aspects will be discussed in this section.

\subsection{}\label{sec:Hein}
The following result is originally due to Hein \cite[Kor.\ 3.11]{HeinDiss}. Since the result
was never published and is crucial for what follows, we include a complete proof for the convenience of the reader. See Remark \ref{rem:hyperell} for a weaker assertion in the case of hyperelliptic linear systems.

\begin{thm}[Hein]\label{thm:Hein}
Let $E$ be a bundle of rank two on a K3 surface $S$. Assume $E$ is $\mu$-semistable of degree zero with respect to a base point free, ample line bundle $H$. 
Then the restriction $E|_C$ to a generic, non-hyperelliptic curve $C\in |H|$ is again semistable.
\end{thm}

\begin{proof} Consider the universal family $$\xymatrix{ |H|&
\ar[l]_-p \kc\ar[r]^-q& S}$$ of curves in the linear system $|H|$. If the restriction $E|_C$ to the generic curve $C\in |H|$ is not semistable, then there exists a relative Harder--Narasimhan
filtration $0\subset F\subset q^\ast E$, i.e.\ $F$ and $F'\coloneqq q^\ast E/F$ are flat over some dense open subset $U\subset |H|$ and the restriction $F|_C$ to a fibre $C\coloneqq\kc_s\subset\kc$, $s\in U$, is a line bundle of positive degree, see \cite[Sec.\ 2.3]{HL} for details and references.
 
 Now, unless $F\subset q^\ast E|_U$ is the pull-back of a sub-line bundle of $E$ on $S$, which would contradict the semistability of $E$, the differential of the induced 
 morphism $\kc|_U\to {\rm Grass}_S(E,1)$ yields
a non-zero map $\kt_q|_U\to F^\ast\otimes F'$. This is the standard argument in the proof  of the Grauert--M\"ulich theorem, cf.\ \cite[Thm.\ 3.1.2]{HL}.
Here, $\kt_q$ is the relative tangent bundle
of the projection $q$. Note that $\kt_q$ restricted to a fibre $C=\kc_s$, $s\in U$, is isomorphic
to the kernel $M_C$ of the
evaluation map $H^0(C,\omega_C)\otimes\ko_C\twoheadrightarrow \omega_C$. 
Indeed, the relative Euler sequence 
$$\xymatrix@C=19pt{0\ar[r]&\ko_\kc\ar[r]&q^\ast\kk\otimes p^\ast \ko(1)\ar[r]&\kt_q\ar[r]&0}$$ of the projective bundle 
$q\colon \kc\cong\PP(\kk)\to S$, where 
$\kk\coloneqq{\rm Ker}\left(H^0(S,H)\otimes \ko_S\twoheadrightarrow H\right)$,
restricted to $C$ is of the form
$\xymatrix@C=19pt{0\ar[r]&\ko_C\ar[r]&\kk|_C\ar[r]&\kt_q|_C\ar[r]&0}$
and, therefore, $M_C \cong\kt_q|_C$.

Hence, the image
of any non-trivial morphism $\kt_q\to F^\ast\otimes F'$ restricted to $C$ would yield a line bundle quotient
$M_C\twoheadrightarrow L$ of degree $\deg(L)\leq \deg(F^\ast)+\deg(F')=-2\deg(F)\leq -2$ on the generic curve $C\in |H|$.

The latter is excluded by  a result of Paranjape and Ramanan \cite[Lem.\ 3.41]{PR} saying that
on a smooth projective curve $C$ the bundle $M_C$  is stable if $C$ is not hyperelliptic, see also \cite{PaDiss}. Indeed, stability in particular implies
that any quotient line bundle $M_C\twoheadrightarrow L$ satisfies $$-2=\frac{2-2g(C)}{g(C)-1}=\mu(M_C)<\deg(L),$$which concludes the proof.\footnote{The last part of the argument replaces the proof of \cite[Kor.\ 3.11]{HeinDiss}, which seems incomplete.
An argument is missing to ensure that the line bundle $L\otimes\omega_C\otimes A^\ast$ is special and, more importantly, effective or weaker that it has non-negative degree.
}
\end{proof}

\begin{remark}
The classical restriction result due to Flenner asserts the semistability of the 
generic restriction
$E|_C$ for $C\in |mH|$ as soon as $m\geq 2(H.H)$ and $H$ is very ample, cf.\ \cite[Thm.\ 7.1.1]{HL}.
So not only that one has to pass to multiples of the linear system, but the bound
in Flenner's theorem depends on the degree of the surface. 
\end{remark}

\begin{remark}\label{rem:hyperell}
When the generic curve $C\in|H|$ is hyperelliptic, the bundle $M_C$
is the pull-back of the kernel of the evaluation map
$\eta\colon H^0(\PP^1,\ko(g-1))\otimes\ko\twoheadrightarrow\ko(g-1)$ on $\PP^1$ under the hyperelliptic quotient $\pi\colon C\twoheadrightarrow \PP^1$. The kernel of $\eta$ is readily identified as
$\ko(-1)^{\oplus g-1}$ and hence $M_C\cong \pi^\ast\ko(-1)^{\oplus g-1}$.
In particular, $M_C$ is really not stable in this case, but it is still polystable.
Therefore, any line bundle quotient $M_C\twoheadrightarrow L$ as in the proof above
satisfies the weaker inequality $\deg(L)\geq -2$. Furthermore, if $\deg(L)=-2$, then
$L\cong\pi^\ast\ko(-1)$.
\end{remark}

We apply the theorem to the tangent bundle $\kt_S$ of a K3 surface $S$
which is $\mu$-stable with respect to any ample line bundle on $S$, see \cite[Sec.\ 9.4]{HuyK3} for a proof and references.

\begin{cor}\label{cor:Hein}
Let $H$ be a base point free, ample line bundle on a K3 surface $S$.
Then for a generic curve $C\in |H|$ either 
$H^0(C,\kt_S|_C)=0$ or there exists a short exact sequence 
$\xymatrix@C=19pt{0\ar[r]&\ko_C\ar[r]&\kt_S|_C\ar[r]&\ko_C\ar[r]&0.}$
In particular, $$\dim H^0(C,\kt_S|_C)\leq2\text{ and } H^0(C,\kt_S|_C(-H))=0.$$
\end{cor}

\begin{proof}
Since $H^0(C,\kt_S|_C)\cong\Hom(\ko_C,\kt_S|_C)$ and $\deg(\kt_S|_C)=0$, stability of the restriction $\kt_S|_C$ would yield $h^0(C,\kt_S|_C)=0$. 

If $\kt_S|_C$ is only semistable, then we have an exact sequence $\xymatrix@C=15pt{0\ar[r]&L_1\ar[r]&\kt_S|_C\ar[r]&L_2\ar[r]&0}$ with $L_1,L_2$ both line bundles of degree zero. In fact, since $\det(\kt_S|_C)\cong\ko_C$, we have $L_1\cong L_2^\ast$. Thus, only two cases can occur. Either, $L_1\cong L_2\cong\ko_C$, in which case $h^0(C,\kt_S|_C)\leq 2$, or both
line bundles, $L_1$ and $L_2$, are non-trivial and hence $h^0(C,\kt_S|_C)=0$. 
By Theorem \ref{thm:Hein}, this concludes the proof in the case of a non-hyperelliptic linear system.

If the generic curve $C\in |H|$ is hyperelliptic and $\kt_S|_C$ is not semistable, then
by Remark \ref{rem:hyperell} the Harder--Narasimhan filtration  
$\xymatrix@C=15pt{0\ar[r]&L_1\ar[r]&\kt_S|_C\ar[r]&L_2\ar[r]&0}$ satisfies $L_2\cong L_1^\ast$ and  $\deg(L)=-2$ for the invertible sheaf $L\coloneqq L_1^\ast\otimes L_2\cong L_1^{-2}$, which is in fact the pull-back $\pi^\ast\ko(-1)$ under the hyperelliptic
quotient $\pi\colon C\to \PP^1$.
Therefore, $\deg(L_1)=1$, $\deg(L_2)=-1$, and $H^0(C,L_1)\cong H^0(C,\kt_S|_C)$.
Hence, either $L_1$ is effective, and then $h^0(C,\kt_S|_C)=h^0(C,L_1)=1$, or it is not, in which
case $h^0(C,\kt_S|_C)=0$.

If $L_1$ is effective, then the Harder--Narasimhan filtration of the restriction $\kt_S|_C$
to the generic curve is of the form $\xymatrix@C=19pt{0\ar[r]&\ko(x_C)\ar[r]&\kt_S|_C\ar[r]&\ko(-x_C)\ar[r]&0}$ for a distinguished point $x_C\in C$. 
Note that $\ko(-2x_C)\cong L\cong\pi^\ast\ko(-1)$ implies that $x_C$ is a fixed point of the hyperelliptic
involution or, in other words, $x_C$ is contained in the ramification curve
$C_0$ of the hyperelliptic quotient $S\to \bar S$.
Consider the rational section of the universal curve ${\mathcal C}\to |H|$
given by mapping $C$ to the distinguished point $x_C$. Then the image of
the induced rational map $\xymatrix{|H|\ar@{-->}[r]&S}$ is 
$C_0$. However,
it is known that $\bar S$ is either $\PP^2$ or one of the Hirzebruch surfaces ${\mathbb F}_i$, $i=0,\ldots,4$, see \cite{Dolgachev,Reid}. Since $C_0$ considered as a curve on $\bar S$ must be contained in $|\omega_{\bar S}^{-2}|$, it cannot be rational. This yields a contradiction. 

The claimed vanishing is a consequence of the above.\end{proof}

\begin{remark}\label{rem:BM} 
 (i)  Note that the vanishing  $H^0(C,\kt_S|_C\otimes H^\ast)=0$ for generic $C$
 can alternatively be deduced from a result  due to Beauville--M\'erindol \cite[Prop.\ 3]{BM}, at least for hyperplane sections: The  normal bundle sequence 
 of any smooth hyperplane section $C\subset S$ is not split unless $C$ is  the fixed locus of an involution. In particular, this does not happen for the generic smooth 
 hyperplane sections and, in fact, can only occur for small genus.
 See Section \ref{sec:HMspectralcurves} for a comparison with the situation for spectral curves.\smallskip
 
 (ii) It turns out that the situation is even better for singular curves.
 More precisely, we have $H^0(C,\kt_S|_C(-C))=0$ for every singular, but reduced curve
 $C\in|H|$ as soon as $g>2$. To prove this, we use the conormal bundle sequence $$\xymatrix{0\ar[r]&\ko_C(-C)\ar[r]&\Omega_S|_C\ar[r]&\Omega_C\ar[r]&0}$$ and
 the natural map $\varphi\colon\Omega_C\to\omega_C$ obtained by
 tensoring $\ko_C(-C)\otimes \Omega_C\to \omega_S|_C\cong\ko_C$, 
 $f\otimes \omega\mapsto df\wedge \omega$, with the line bundle $\ko_C(C)$.
 Clearly, $\ker(\varphi)$ and $\coker(\varphi)$ are concentrated in
 the singular points of $C$ of which there are at most $g$.  In fact, it is known classically that $h^0(\ker(\varphi))=h^0(\coker(\varphi))\leq 2(g-g(\tilde C))$, where $\tilde C\to C$ is the normalization, cf.\ \cite{DiGr,BuGr}. In particular, $0<h^0({\rm Ker}(\varphi))\leq 2g$.

Consider $0\ne s\in H^0(C,\kt_S|_C(-C))$ interpreted as a homomorphism
$s\colon \ko_C(C)\to \kt_S|_C\cong\Omega_S|_C$. Since $\omega_C\cong\ko_C(C)$ and $\coker(\varphi)\ne0$, the composition of $s$ with $\Omega_S|_C\to\Omega_C\to \omega_C$ has to be trivial. In other words, the composition $\ko_C(C)\to\Omega_S|_C\to \Omega_C$ takes image in the torsion
subsheaf $\ker(\varphi)\subset\Omega_C$.

Hence, $s$ induces a non-trivial homomorphism
$\ko_C(C)\otimes \ki_{Z}\to \ko_C(-C)$:
$$\xymatrix{\ko_C(C)\otimes\ki_Z\ar[r]\ar@{-->}[d]&\ko_C(C)\ar[d]_s\ar[r]&\ker(\varphi)\ar@{^(->}[d]\\
\ko_C(-C)\ar[r]&\Omega_S|_C\ar[r]&\Omega_C}$$ 
Here, $Z\subset C$ is contained in the singular locus of $C$ with $h^0(\ko_Z)\leq h^0(\ker(\varphi))$. Together with $(2C.C)=4g-4$ this contradicts $h^0(\ker(\varphi))\leq 2g$ for $g>2$.
\end{remark}

\subsection{} 
A related result of Balaji and Koll\'ar \cite[Thm.\ 2]{BaKo}
translated to our situations gives the stability of $\kt_S|_C$ for a generic
$C \in |mH|$, $  m\geq 6$,  and $|H|$ a base point free, ample, non-hyperelliptic linear system on $S$. As observed in the proof of Corollary \ref{cor:Hein}, stability has the following direct consequence.

\begin{cor}\label{cor:BK}
Assume $|H|$ is a base point free, ample, non-hyperelliptic linear system
on a K3 surface $S$. Then
for a generic curve $D\in |mH|$, $m\geq 6$,  the restriction $\kt_S|_D$ is stable
and, in particular, $H^0(D,\kt_S|_D)=0$.\qed
\end{cor}

Neither \cite{BaKo} nor the above consequence will be used in what follows. Nonetheless it emphasises our point that multiple linear systems behave better than primitive ones.

\begin{remark}\label{rem:BogoHein}
(i) One should compare their stability result to other ones. For instance, a result of Bogomolov, cf.\ \cite[Thm.\ 7.3.5]{HL},  implies that $\kt_S|_C$ is stable for any smooth  curve $C \in |mH|$ for $m \geq 96$ and $|H|$ a very ample linear system on $S$. The
bound is significantly bigger than that of \cite{BaKo}. A better bound is due to Hein
 \cite[Thm.\ 2.8]{Hein2}; namely the restriction $\kt_S|_C$ to any smooth curve $C\in |mH|$ is stable if $2m(H.H) > 96$, which holds for all $m \geq 1$ as soon as $(H.H) > 48$.
 \smallskip
 
(ii)  For smaller degree there exist examples of positive-dimensional isotrivial subfamilies of curves $C\in |H|$ with non-stable restriction $\kt_S|_C$, see \cite{GounelasOttem} and
 Example \ref{ex:Fermat}.
\end{remark}

\begin{remark}
The vanishing $H^0(D,\kt_S|_D)=0$ for $D\in |\ko(m)|$, $m\geq 5$, also follows from  \cite[Thm.\ 5.1]{Totaro}, see Remark \ref{rem:Totaro}.
\end{remark}

\subsection{}  The two results, Theorem \ref{thm:Hein} and Corollary \ref{cor:BK}, are
complemented by the following, which is the key to the proof of the main result Theorem
\ref{thm:main}.

\begin{thm}\label{thm:main2}
Assume $|H|$ is a base point free, ample linear system
on a K3 surface $S$ such that the locus of non-reduced curves in $|H|$
has codimension at least three. Then, for a generic
curve $D\in |mH|$, $m\geq 2$, one has $H^0(D,\kt_S|_D)=0$.
\end{thm}

\begin{proof} For the case $g=2$ see Proposition \ref{prop:Yagna}, so we restrict to the
case $g>2$ and hence can make use of Remark \ref{rem:BM}, (ii).
We first show that it suffices to prove the assertion for $m=2$.

Indeed, assume $H^0(D,\kt_S|_D)=0$ for the generic curve $D\in |2H|$. Then, if $m>2$, pick a generic curve
$D'\in |(m-2)H|$ and tensor the exact sequence $$\xymatrix@C=19pt{0\ar[r]&\ko_{D'}(-D)\ar[r]&\ko_{D\cup D'}\ar[r]&\ko_{D}\ar[r]&0}$$ with $\kt_S$. The resulting
long exact cohomology sequence is of the form $$\xymatrix@C=19pt{0\ar[r]&H^0(D',\kt_S|_{D'}(-D))\ar[r]&H^0(D\cup D',\kt_S|_{D\cup D'})\ar[r]&H^0(D,\kt_S|_{D})\ar[r]&\cdots.}$$
By Corollary \ref{cor:Hein}, the restriction $\kt_S|_{D'}$ satisfies $H^0(D',\kt_S|_{D'}(-D))=0$. Combined with the assertion for $m=2$, this yields the vanishing $H^0(D\cup D',\kt_S|_{D\cup D'})=0$ and, by semi-continuity, this proves the assertion for the generic
curve in $|mH|$.

To prove the assertion for $m=2$, let $C,C'\in |H|$ be generic curves.
Again by semi-continuity, it suffices to show
$H^0(C\cup C',\kt_S|_{C\cup C'})=0$.

 If the restriction $\kt_S|_C$ to a generic  $C\in |H|$
satisfies $H^0(C,\kt_S|_C)=0$, then also $H^0(C\cup C',\kt_S|_{C\cup C'})=0$, by the same arguments as above. Otherwise, by Corollary \ref{cor:Hein}, 
there exists a short exact sequence 
\begin{equation}\label{eqn:sesTS}
\xymatrix@C=19pt{0\ar[r]&\ko_{C}\ar[r]&\kt_S|_{C}\ar[r]&\ko_{C}\ar[r]&0.}
\end{equation} This leaves us with the two cases: (i) $h^0(C,\kt_S|_C)=1$,
or equivalently (\ref{eqn:sesTS}) does not split,  or (ii) 
$h^0(C,\kt_S|_C)=2$, i.e.\ $\kt_S|_C\cong\ko_C\oplus\ko_C$, for the generic curve $C\in |H|$.

\smallskip

(i) We assume $h^0(C,\kt_S|_C)=1$ for a generic $C\in |H|$ and suppose
that also $h^0(C\cup C',\kt_S|_{C\cup C'})=1$ for two generic curves
$C,C'\in |H|$. To derive a contradiction we pick a generic pencil $p\colon\kc   \to \PP^1$ in the base point free linear system
$|H|$ and fix a generic curve $C'\in |H|$. By our assumption on the non-reduced locus in $|H|$, all fibres of $p$ are
reduced. The other projection $q\colon \kc\to S$ is the blow-up in the base locus of the pencil.
By Remark \ref{rem:BM}, (ii),  the natural restriction map
$H^0(C\cup C',\kt_S|_{C\cup C'})\,\hookrightarrow H^0(C',\kt_S|_{C'})$ is injective
for all fibres $C=\kc_s$, $s\in \PP^1$, and under our assumption in fact bijective.
Similarly,  also the restriction maps
$H^0(C\cup C',\kt_S|_{C\cup C'})\,\hookrightarrow H^0(C,\kt_S|_{C})$ are injective.
Thus, the bundle with fibres  $H^0(C\cup C',\kt_S|_{C\cup C'})$ is isomorphic
to the trivial line bundle with constant fibre $H^0(C',\kt_S|_{C'})$ and, at the same time,
it is contained
in $p_\ast q^\ast\kt_S$, which generically has fibre $H^0(C,\kt_S|_C)$. Thus,
$\ko_{\PP^1}\subset p_\ast q^\ast \kt_S$, which yields the contradiction 
$k=H^0(\PP^1,\ko_{\PP^1})\subset H^0(\PP^1,p_\ast q^\ast\kt_S)\cong H^0(\kc,q^\ast\kt_S)\cong H^0(S,\kt_S)=0$.

To formalize the argument, let $\tilde p\colon\tilde\kc\to \PP^1$ be the family
obtained by gluing $\kc$ and the constant family $\kc'\coloneqq C'\times\PP^1$ along $C'$. Here, $C'$ is naturally embedded into $\kc$ by viewing
$\kc$ as the blow-up of $S$ and into $\kc'=C'\times\PP^1$ as the graph
of the induced projection $C'\subset\kc\to \PP^1$.  So, if $C$ is the fibre $\kc_s=p^{-1}(s)$,  $s\in\PP^1$, then
 $\tilde p^{-1}(s)\cong C\cup C'$. 
We denote by $\tilde q\colon \tilde \kc\to S$ the second projection. Then the fibres of $E\coloneqq\tilde p_\ast\tilde q^\ast\kt_S$ are the lines
$H^0(C\cup C',\kt_S|_{C\cup C'})$. Moreover, we have a natural
inclusion $E\subset p_\ast q^\ast\kt_S$, which is the direct image of the restriction $\tilde q^\ast \kt_S\to q^\ast\kt_S$ under the inclusion $\kc\subset\tilde\kc$.
Via the restriction to $\kc'\subset\tilde\kc$, the sheaf $E$ can also be identified
with the trivial
bundle $H^0(C',\kt_S|_{C'})\otimes\ko_{\PP^1}\cong\ko_{\PP^1}$.
\smallskip

(ii) Let us now assume that $\kt_S|_C\cong\ko_C\oplus\ko_C$ for the generic curve
$C\in |H|$. As above, we choose a generic pencil $\kc\to\PP^1$
 and let $C'\in |H|$ be a fixed generic curve.
 As before, it is enough to show that $h^0(C\cup C',\kt_S|_{C\cup C'})=0$ for the generic fibre $C=\kc_s$, $s\in \PP^1$. Due to Remark
\ref{rem:BM}, (ii),   
for every fibre $C=\kc_s$ we have natural inclusions $$H^0(C\cup C',\kt_S|_{C\cup C'}) \subset H^0(C,\kt_S|_C)\text{ and
}H^0(C\cup C',\kt_S|_{C\cup C'}) \subset H^0(C',\kt_S|_{C'})\cong k^2.$$ 
Hence, $h^0(C\cup C',\kt_S|_{C\cup C'})\leq2$. Furthermore, if $h^0(C\cup C',\kt_S|_{C\cup C'})=2$
holds generically, then it holds for all fibres $C$. This, however, would
result in the inclusion of the trivial bundle $\ko_{\PP^1}^{\oplus 2}$ with
fibre at the point $s\in \PP^1$ naturally identified with
$H^0(\kc_s\cup C',\kt_S|_{\kc_s\cup C'}) \cong H^0(C',\kt_S|_{C'})\cong k^2$ into the sheaf $p_\ast q^\ast\kt_S$ and eventually lead to a similar contradiction
$k^2\cong H^0(\PP^1,\ko_{\PP^1}^{\oplus 2})\subset H^0(\PP^1,p_\ast q^\ast\kt_S)\cong H^0(\kc,q^\ast\kt_S)\cong H^0(S,\kt_S)=0$.

\smallskip

(iii) It remains to deal with the case $h^0(C,\kt_S|_C)=2$ and
$h^0(C\cup C',\kt_S|_{C\cup C'})=1$ for generic curves $C,C'\in |H|$.
Here, we use the assumption that the
non-reduced curves in $|H|$ form a closed set of codimension at least three.
We denote its  open complement of all reduced curves by $U\coloneqq |H|_{\rm red}\subset |H|$.

Now, under our assumptions, we can consider the rational map $$\xymatrix{\psi\colon|H|\ar@{-->}[r]&\PP(H^0(C',\kt_S|_{C'})),} \xymatrix@C=15pt{C\ar@{|->}[r]& H^0(C\cup C',\kt_S|_{C\cup C'}).}$$ 
The map is regular on the complement $U\setminus Z_{C'}$ of the proper, closed subset $Z_{C'}\subset U$
of all curves $C\in U$ with $H^0(C\cup C',\kt_S|_{C\cup C'})\to H^0(C',\kt_S|_{C'})$
surjective (and hence bijective by Remark \ref{rem:BM}, (ii)). We proceed by distinguishing two cases:
Either, ${\rm codim} (Z_{C'})\geq 2$ or ${\rm codim} (Z_{C'})=1$.
In the first case, the generic fibre $\psi^{-1}(\ell)$ of $\psi\colon U\setminus Z_{C'}\to \PP(H^0(C',\kt_S|_{C'}))$ contains a complete, integral curve, say $T\subset \psi^{-1}(\ell)$. We then consider
the restriction $\kc_T$ of the universal curve to $T$ and its two projections
$\xymatrix@C=18pt{T&\ar[l]_-p\kc_T\ar[r]^-q&S.}$
As in step (ii),  $E$ on $T$ shall denote the bundle with fibre $H^0(\kc_t\cup C',\kt_S|_{\kc_t\cup C'})$ over the point $t\in T$. By our choice of $T$, this bundle is isomorphic
to $\ell\otimes\ko_T$ and hence trivial. On the other hand, it is a subsheaf
of $p_\ast q^\ast\kt_S$ and, therefore, $H^0(\kc_T,q^\ast\kt_S)\ne0$. In (ii), the total
space of the one-dimensional family was the blow-up of $S$ which was enough to
derive the contradiction $H^0(S,\kt_S)\ne0$. The same argument works here
due to Lemma \ref{lem:pullbackT} below.

To conclude we have to deal with the case that ${\rm codim}(Z_{C'})=1$. In this
case, $Z_{C'}$ itself contains a complete, integral curve $T\subset Z_{C'}$,
simply because the boundary $\bar Z_{C'}\setminus Z_{C'}$ of the closure
$\bar Z_{C'}\subset |H|$ is contained in $|H|\setminus |H|_{\rm red}$, which by assumption is of dimension at most $g-3$ and hence ${\rm codim}_{\bar Z_{C'}} (\bar Z_{C'}\setminus Z_{C'})\geq 2$.
 For the
corresponding family $\kc_T$ as above, the bundle with fibres $H^0(\kc_t\cup C',\kt_S|_{\kc_t\cup C'})\cong H^0(C',\kt_S|_{C'})$ is trivial. Hence, $p_\ast q^\ast\kt_S$ contains
$\ko_T^{\oplus 2}$ which again yields a contradiction to $H^0(S,\kt_S)=0$.
\end{proof}

\begin{lem}\label{lem:pullbackT}
Assume $q\colon S'\to S$ is a generically finite morphism
from a normal projective irreducible surface $S'$ onto a K3 surface $S$.
Then $H^0(S',q^\ast\kt_S)=0$.
\end{lem}

\begin{proof} Using Stein factorization, we can reduce to the case that $q$ is finite.
Then use that for a finite morphism $q\colon S'\to S$ the pull-back $q^\ast\kt_S$
is $\mu$-polystable, cf.\ \cite[Lem.\ 3.2.3]{HL}. Clearly, if $q^\ast\kt_S$ is $\mu$-stable,
then $H^0(S',q^\ast\kt_S)=0$. Otherwise, $q^\ast\kt_S$
is isomorphic to a sum of invertible sheaves, say $q^\ast\kt_S\cong L_1\oplus L_2$,
which satisfy $L_1\cong L_2^\ast$. There are two cases: Either, $L_1$ and $L_2$
are not trivial, in which case $H^0(S',L_i)=0$ and hence $H^0(S',q^\ast \kt_S)=0$, or
$L_1\cong L_2\cong\ko_{S'}$, which contradicts ${\rm c}_2(q^\ast\kt_S)=q^\ast {\rm c}_2(\kt_S)\ne0$.
\end{proof}

\begin{remark}\label{rem:m=3}
Without the assumption on the reduced locus $|H|_{\rm red}\subset|H|$ the arguments 
in the proof show that 
$h^0(C, T_S|_C) \leq 1$ for generic $C\in |2H|$. This renders case (iii), which is where we used the assumption, in the proof unnecessary. Running the other two cases again for $C\in |2H|$ and $C'\in |H|$, we obtain $H^0(D, \kt_S|_D) = 0$ for $D = C\cup C'$ and hence also for generic curves $D\in |3H|$.
\end{remark}

\section{Maximal variation}\label{sec:maxvar}

We study the family $\kc\to|H|$ given by the universal curve over a base point
free linear system $|H|$ on a K3 surface $S$. We denote by
 $|H|_{\rm sm}\subset|H|$ the subset of all smooth curves, which
 is  Zariski open and dense.
 
Mapping a point $s\in |H|$  corresponding to a smooth curve $C\coloneqq\kc_s$ to the corresponding point $[C]\in M_g$ in
the  Deligne--Mumford stack of smooth curves of genus $g$, where $2g-2=(H.H)^2=d$,
describes a rational map
$$\xymatrix{\Phi\colon|H|\ar@{-->}[r]&M_g}.$$
The map is regular on the open subset $|H|_{\rm sm}\subset |H|$ of all smooth curves. The curves in $|H|$ have \emph{maximal variation} if this map is generically quasi-finite or, equivalently, if its image is of dimension $g=\dim|H|$.

\subsection{}\label{sec:fibredim}
Consider a smooth fibre $C\coloneqq \kc_s$,  $s\in |H|_{\rm sm}$.
The tangent spaces of $|H|$ at $s$ and of $M_g$ at $[C]=\Phi(s)$ 
are naturally isomorphic to $H^0(C,\omega_C)\cong H^0(C,H|_C)$
and  $H^1(C,\kt_C)$. The differential $d\Phi$ at the point $s$ is the boundary
map $$d\Phi\colon H^0(C,H|_C)\to H^1(C,\kt_C)$$ of the long exact cohomology sequence
associated with the normal bundle sequence \begin{equation}\label{eqn:normses}
\xymatrix@C=18pt{0\ar[r]&\kt_C\ar[r]&\kt_S|_C\ar[r]&\omega_C\cong H|_C\ar[r]&0.}
\end{equation}

In particular,  using $H^0(C,\kt_C)=0$, the relative Zariski tangent  space $\kt_\Phi(s)$ of $\Phi$ at the point
$s\in |H|_{\rm sm}$ is the vector space: $$\kt_\Phi(s)\cong H^0(C,\kt_S|_C).$$

\begin{lem}\label{lem:tangent}
The rational map $\xymatrix{\Phi\colon |H|\ar@{-->}[r]&M_g}$ is generically quasi-finite if and only if $H^0(C,\kt_S|_C)=0$ for the generic curve $C\in|H|$.
More generally, the dimension of the generic fibre of $\Phi$ is bounded from above by $h^0(C,\kt_S|_C)$ for any (generic) smooth curve $C\in|H|$.
\end{lem}

\begin{proof}
For a smooth curve $C=\kc_s$, $s\in |H|$, 
 the differential $d\Phi\colon H^0(C,H|_C)\to H^1(C,\kt_C)$,
 is injective if and only if $H^0(C,\kt_S|_C)=0$. 
More generally, as we work in characteristic is zero, the dimension of a generic fibre $\Phi^{-1}([C])$ is the dimension of ${\rm Ker}(d\Phi)$ at the generic point of that fibre. In particular, the generic fibre is zero-dimensional if and only if $H^0(C,\kt_S|_C)=0$ for the generic curve $C=\kc_s$.
%
\end{proof}

\begin{cor}\label{cor:dim2}
Assume $|H|$ is a base point free, ample linear system on a K3 surface $S$. Then the generic fibre of 
the rational map $\xymatrix{\Phi\colon |H|\ar@{-->}[r]&M_g}$ is of dimension
at most two. Equivalently,
$$\dim{\rm Im}(\Phi)\geq g-2.$$
\end{cor}

\begin{proof}
This is a direct consequence of Corollary \ref{cor:Hein} and Lemma \ref{lem:tangent}.
\end{proof}

In particular, the family of smooth curves in $|H|$ can never be isotrivial for $g>2$. For
$7\leq g\leq 23$ this can also be deduced from \cite{Benzo}.

As another immediate consequence of the description of the relative Zariski tangent space, we state the following result which is strengthened by the main result Theorem \ref{thm:main}.

\begin{cor}\label{cor:BKcons}
Assume $H$ is a base point free, ample, non-hyperelliptic linear system on a K3 surface. Then the rational
map $$\Phi_m\colon |mH|_{\rm sm}\to M_{g_m},$$
$g_m\coloneqq\dim |mH|$, is generically quasi-finite for all $m\geq 6$. Equivalently, $$\dim{\rm Im}(\Phi_m)=g_m=\frac{m^2}{2}\cdot(H.H)+1.$$
\end{cor}

\begin{proof}
For $m\geq 6$ this follows from the result of Balaji--Koll\'ar, see Corollary \ref{cor:BK}.
\end{proof}

The result holds more generally for $m \geq 5$ and without the non-hyperelliptic assumption due to  Bott vanishing \cite[Lem.\ 3.7 \&  Rem.\ 3.8]{Totaro}. We discuss the connection of this vanishing to the vanishing of $H^0(C,\kt_S|_C)$ in Section  \ref{sec:CompBott}.

\smallskip

The next result is proved by techniques similar to the ones in the proof of Theorem \ref{thm:main2}. It can be seen as the analogue of \cite[Lem.\ 3.5]{Totaro}, see Remark \ref{rem:Totaro}.

\begin{lem}\label{lem:vanishmult}
Assume $|H|$ is a base point free, ample linear system such that $H^0(C,\kt_S|_C)=0$
for the generic curve $C\in |H|$. Then $H^0(D,\kt_S|_D)=0$ for the generic
curve $D\in |mH|$ for every $m\geq 1$.
\end{lem}

\begin{proof}
We explain the case $m=2$. The general case is similar. Pick a curve $C\in |H|$ with $H^0(C,\kt_S|_C)=0$ and consider
the non-reduced curve $D=2C\in |2H|$. Twisting the short exact sequence
$\xymatrix@C=19pt{0\ar[r]&\ko_C(-C)\ar[r]&\ko_D\ar[r]&\ko_C\ar[r]&0}$
with the locally free sheaf $\kt_S$, we obtain an exact sequence
$\xymatrix@C=19pt{0\ar[r]&H^0(C,\kt_S|_C(-C))\ar[r]&H^0(D,\kt_S|_D)\ar[r]&H^0(C,\kt_S|_C)\ar[r]&\cdots}$. Since $H^0(C,\kt_S|_C)$ and hence $H^0(C,\kt_S|_C(-C))$
are both trivial, the vanishing of $H^0(D,\kt_S|_D)$ follows.
By semi-continuity, the vanishing then holds for the generic $C\in |2H|$.
\end{proof}

\subsection{}\label{sec:proof} We now come to the proof of Theorem \ref{thm:main}, 
which, having set up all the necessary machinery, is now almost immediate.

Assume $|H|$ is a base point free, ample linear system  with a non-reduced locus of codimension at least three. Then, by virtue of Theorem \ref{thm:main2}, we know that $H^0(D,\kt_S|_D)=0$ for a generic curve $D\in |mH|$, $m\geq 2$. Applying Lemma \ref{lem:tangent}
concludes the proof. \qed

\begin{remark}
Note that by virtue of Remark \ref{rem:m=3}, maximal variation holds for $|mH|$, $m\geq 3$,
without any assumption on the non-reduced locus. Also note that the assumption on the non-reduced locus in $|H|$ is a Zariski open condition for polarized K3 surfaces $(S,H)$.
It always holds for polarized K3 surfaces with  $\Pic(S)=\ZZ\cdot H$.
\end{remark}

\begin{ex} The two assumptions on the linear system $|H|$ to be base point free and ample are both necessary.

(i) Let $\pi\colon S\to \PP^1$ be an isotrivial elliptic K3 surface, then
$H=\pi^\ast \ko(1)$ is base point free but not ample. In this case,  none of the linear systems $|mH|$, $m\geq 1$, has maximal variations.
\smallskip

(ii) If $|H|$ is ample but not base point free, then $H\cong\ko((g-1)E+C)$.
Here, $E$ is a smooth elliptic curve and $C\cong\PP^1$ with $(C.E)=1$, see \cite[Cor.\ 2.3.15]{HuyK3}. The general member in $|H|$ is a curve of the form 
$E_1+\cdots+E_{g-1}+C$. For example, for $g=4$ this provides a family of dimension at most three of abstract curves in the four-dimensional linear system $|H|$.
\end{ex}

\begin{ex}\label{ex:Fermat}
Similarly, the conclusion of Theorem \ref{thm:main} cannot be strengthened.
In general, neither is the map $|H|_{\rm sm}\to M_g$ quasi-finite nor generically injective.
For a concrete example, consider the Fermat quartic $S\subset\PP^3$ given by the equation $x_0^4+\cdots+x_3^4=0$. Then the hyperplane sections
$x_0=t x_1$ describe a one-dimensional isotrivial family of smooth curves in $|\ko(1)|$.
Due to the non-trivial automorphism group of $(S,\ko_S(1))$, namely
$\Aut(S,\ko_S(1))=\mu_4^4/\mu_4\rtimes {\mathfrak S}_4$, the map
$|\ko_S(1)|_{\rm sm}\to M_3$ cannot be generically injective.\end{ex}

\subsection{}\label{sec:CompBott}
Let us briefly spell out the relation between the vanishing
of $H^0(C,\kt_S|_C)$ and that of $H^1(S,\kt_S\otimes H^\ast)$.
Note that while the former is the relative tangent space of the map $\xymatrix{\Phi\colon|H|\ar@{-->}[r]& M_g}$ 
at $C\in |H|$ the latter is the relative tangent space of the map $$\tilde\Phi\colon(C\subset S)\mapsto C$$ with varying K3 surface $S$, cf.\ \cite[Sec.\ 5.2]{BeauvilleFano}.

From  $\xymatrix@C=18pt{0\ar[r]&\kt_S\otimes H^\ast\ar[r]&\kt_S\ar[r]&\kt_S|_C\ar[r]&0}$ and  $H^0(S,\kt_S)=0= H^2(S,\kt_S)$ we obtain the long exact cohomology sequence
$$\xymatrix@C=17pt{0\ar[r]&H^0(\kt_S|_C)\ar[r]&H^1(\kt_S\otimes H^\ast)\ar[r]&H^1(\kt_S)\ar[r]&H^1(\kt_S|_C)\ar[r]&H^2(\kt_S\otimes H^\ast)\ar[r]&0.}$$ 

Hence, if $H^1(S,\kt_S\otimes H^\ast)=0$, then also $H^0(C,\kt_S|_C)=0$, but the converse need not hold. This immediately leads to the following observation.

\begin{lem}\label{lem:BVqf}
Assume $H^1(S,\kt_S\otimes H^\ast)=0$. Then the morphism $|H|_{\rm sm}\to M_g$ is
quasi-finite.
\end{lem}

\begin{proof}
Note that the vanishing only depends on the linear system and not on the individual curve $C\in |H|$. It implies $H^0(C,\kt_S|_C)=0$ for all  curves
$C\in |H|_{\rm sm}$. Therefore, $\Phi\colon |H|_{\rm sm}\to M_g$ is immersive and, hence, quasi-finite.
\end{proof}

\begin{remark}\label{rem:Totaro}
(i) Totaro \cite{Totaro} calls $H^1(S,\kt_S\otimes H^\ast)=0$,
or dually $H^1(S,\Omega_S\otimes H)=0$,  \emph{Bott vanishing}.  If $|H|$ is base point free and ample, then Bott vanishing
for $H$ implies Bott vanishing
for all multiples $mH$, $m\geq 1$, which in particular implies $H^0(D,\kt_S|_D)=0$ for every curve $D\in |mH|$. In this sense, Lemma \ref{lem:vanishmult}
really is the analogue of \cite[Lem.\ 3.5]{Totaro}.

 Bott vanishing holds if the following two conditions
are satisfied: $(H.H)\geq 74$ and there is no curve
$E$ with $(E.E)=0$ and  $1\leq (H.E)\leq 4$, see \cite[Thm.\ 5.1]{Totaro}.
For example, these assumptions are met for the linear system $|mH|$ if $m\geq 5$ and $(H.H)\geq 4$. 
 Note that there exist examples where Bott vanishing fails for arbitrary large $(H.H)$, 
  see \cite[Sec.\ 6]{Totaro}.
\smallskip

(ii) Most of the results in \cite{Totaro} concern K3 surfaces of Picard number one.
For example, under this assumption, \cite[Thm.\ 3.3]{Totaro} asserts Bott vanishing
$H^1(S,\kt_S\otimes H^\ast)=0$ for $(H.H)=20$ or $(H.H)\geq 24$. In particular, under these assumptions also the primitive linear system $|H|$ 
and hence all its multiples $|mH|$ have maximal variation, see \cite[Lem.\ 3.5]{Totaro}. On the other hand, Bott vanishing definitely fails for $(H.H)=22$, see \cite[Thm.\ 3.2]{Totaro}, but for the generic
K3 surface $|H|$ nevertheless has maximal variation.
\smallskip

\smallskip

(iii) The dimension $h^1(S,\kt_S\otimes H^\ast)$ is linked to the corank of the Wahl map
$${\bigwedge}^2H^0(C,\omega_C)\to H^0(C,\omega_C^3),$$ see \cite{CDS} for details.
Recall that for a smooth curve $C\subset S$ in a K3 surface the Wahl map is not surjective, cf.\ \cite{BM}.
\end{remark}

\section{Comparison with the Hitchin system}\label{sec:Hitchin}

Instead of passing from a (smooth) curve $C\in |H|$ in a K3 surface $S$ to 
non-primitive curves $D\in |mH|$ we now look at spectral curves $D\to C$ of degree $m$. There are many similarities between the two situations. 
To some extent the space $V(\omega_C)$, where the spectral curves live, is in a way easier and more uniform that K3 surfaces, even though the stability of the restriction of the tangent bundle fails.
We will also explain how
to use the information for spectral curves to deduce partial results for K3 surfaces more directly.

\subsection{}\label{sec:HMspectralcurves}
 Let us begin by recalling some basic facts. A \emph{spectral curve} of degree $m$ is naturally associated to $(s_i)\in \bigoplus_{i=1}^mH^0(C,\omega_C^i)$. More precisely, given $(s_i)$ one obtains a curve
$$D\coloneqq D_{(s_i)}\subset V\coloneqq {\mathbb V}(\omega_C)\subset T\coloneqq{\mathbb P}(\ko_C\oplus\omega_C)$$ in the linear system $|\pi^\ast\omega^m_C\otimes\ko_\pi(m)|$, where $\pi\colon T\to C$ is the natural projection. Indeed, $H^0(T,\pi^\ast\omega^m_C\otimes\ko_\pi(m))\cong H^0(C,\omega^m_C\otimes S^m(\ko_C\oplus\omega^\ast_C))\cong \bigoplus_{i=0}^m H^0(C,\omega_C^i)$.
More explicitly, using the canonical sections $s$ and $t$ of $\pi_*\ko_\pi(1)\cong\ko_C\oplus \omega_C^\ast$ and $\pi_\ast(\pi^\ast\omega_C\otimes \ko_\pi(1))\cong \omega_C\oplus\ko_C$ allows
one to write an explicit equation for $D=D_{(s_i)}$:
\begin{equation}\label{eqn:BNR}
t^m+\pi^\ast s_1\cdot t^{m-1}\cdot s+\cdots+\pi^\ast s_m \cdot s^m\in H^0(T,\pi^\ast\omega_C^m\otimes \ko_\pi(m)),
\end{equation} cf.\ \cite[Sec.\ 3]{BNR}.
For example, the zero section $C\subset V\subset T$ corresponds to $(s_1=0)\in H^0(C,\omega_C)$. Also, recall that $\pi_\ast\ko_D\cong\ko_C\oplus\cdots\oplus \omega_C^{-m+1}$ for a spectral curve $\pi\colon D\to C$ of degree $m$ and observe that
the restriction of $\ko_\pi(1)$ to $V\subset T$ is trivial.

The surface $T$ or rather the open set $V\subset T$ should be thought of as a
replacement for the K3 surface $S$. 
For example, $V$ has a natural symplectic structure. 
Furthermore, the choice of $C\in |H|$ induces a filtration $0\subset H^0(S,\ko)\subset H^0(S,H)\subset\cdots\subset H^0(S,mH)$ with quotients
$H^0(C,\omega_C^i)$, $i=0,\ldots, m$. The filtration specialises to the direct sum
$\bigoplus_{i=0}^m H^0(C,\omega_C^i)$. 

The tangent bundle $\kt_V$ of $V$ can be described naturally as an extension
of pull-back line bundles \begin{equation}\label{eqn:tangV}\xymatrix@C=15pt{0\ar[r]&\pi^\ast\omega_C\ar[r]&\kt_V\ar[r]&\pi^\ast\kt_C\cong\pi^\ast\omega^\ast_C\ar[r]&0.}\end{equation} The sequence splits when restricted to the zero-section $C\subset V$.
Contrary to the case of K3 surfaces, the symplectic surface $V$ admits global vector fields. More precisely, $H^0(V,\kt_V)\cong H^0(C,\pi_\ast\kt_V)$ which can be computed 
via the direct image 
$$\xymatrix@C=15pt{0\ar[r]&\pi_\ast\pi^\ast\omega_C\cong \omega_C\otimes\bigoplus_{i\leq 0}\omega_C^i\ar[r]&\pi_\ast\kt_V\ar[r]&\pi_\ast\pi^\ast\kt_C\cong\omega^\ast_C\otimes\bigoplus_{i\leq0}\omega_C^i\ar[r]&0}$$
 of (\ref{eqn:tangV}) as $$H^0(V,\kt_V)\cong H^0(C,\ko_C)\oplus H^0(C,\omega_C).$$
In fact, thinking of $H^0(V,\kt_V)$ as the tangent space of ${\rm Aut}(V)$ both parts can be explained geometrically: (i) The vector bundle $V$ comes with a natural $k^\ast$-action; (ii) Any section $u\in H^0(C,\omega_C)$ acts by translation on $V$. Combining the two, one finds that
$(\lambda,u)\in k^\ast\times H^0(C,\omega_C)$ acts on $V$ by $v\mapsto \lambda v+u$.
From this, one also obtains an action of $k^\ast\times H^0(C,\omega_C)$ on the
space  of spectral covers or, equivalently, on the parameter space $\bigoplus_{i=1}^mH^0(C,\omega_C^i)$. Concretely, $k^\ast$ acts by
 $\lambda\cdot(s_i)=(\lambda^i\cdot s_i)$
and $H^0(C,\omega_C)$  by applying Tschirnhaus transformation to (\ref{eqn:BNR}).

\smallskip

As in the case of K3 surfaces, we are interested in the restriction $\kt_V|_D$ to a
smooth spectral curve $D\to C$.  The restriction of vector fields on $V$ to $D$ can be understood in terms of the exact sequence
$$\xymatrix@C=19pt{0\ar[r]&H^0(V,\kt_V(-D))\ar[r]&H^0(V,\kt_V)\ar[r]^-r&H^0(D,\kt_V|_D)\ar[r]&H^1(V,\kt_V(-D))\ar[r]&\cdots,}$$
which implies the injectivity  of $r$, as $H^0(V,\kt_V(-D))\cong H^0(V,\kt_V\otimes\pi^\ast\omega_C^{-m})=0$ for $m>1$. In fact, $r$ is bijective according
to the next result.

\begin{lem}\label{lem:HM}
Let $D\to C$ be a spectral curve of degree $m\geq 2$. Then the restriction map is an isomorphism
$$\xymatrix{H^0(C,\ko_C)\oplus H^0(C,\omega_C)\cong H^0(V,\kt_V)\ar[r]^-\sim&
 H^0(D,\kt_V|_D).}$$
 In particular, the restriction $\kt_V|_D$ is not semistable and $h^0(D,\kt_V|_D)=g(C)+1$, 
\end{lem}

\begin{proof} We give an alternative argument to \cite[Sec.\ 3]{HM}. Restricting the exact sequence (\ref{eqn:tangV}) to $D$ and then taking direct image under $\pi\colon D\to C$, we obtain the exact sequence
$$\xymatrix{0\ar[r]&\omega_C\otimes\bigoplus_{i=-m+1}^0\omega_C^i\ar[r]&\pi_*(\kt_V|_D)\ar[r]&\omega_C^\ast\otimes\bigoplus_{i=-m+1}^0\omega_C^i\ar[r]&0.}$$
Taking cohomology yields the result.
\end{proof}


Despite not giving the desired vanishing even for large $m$ and generic $D$, 
compare with Corollary \ref{cor:Hein} and Theorem \ref{thm:main2}, the observation is useful
 as the bound is independent of $m$ and, for example, stronger than the
 obvious one $h^0(D,\kt_V|_D)\leq g_m\coloneqq g(D)$ for $m\geq 2$, cf.\ Remark \ref{rem:BM}.

The analogue of Theorem \ref{thm:main} in the case of spectral curves is the following consequence, cf.\ \cite[Thm.\ 3.1]{HM}.

\begin{cor}[Hodge--Mulase]\label{cor:HM}
Let $C$ be a smooth curve of genus $g$. We view the space $\bigoplus_{i=1}^mH^0(C,\omega_C^i)$ as the  variety of spectral covers $D\to C$ of degree $m$
and consider the rational map
\begin{equation}\label{eqn:HM}\xymatrix{
\Phi\colon\bigoplus_{i=1}^mH^0(C,\omega_C^i)
 \ar@{-->}[r]&M_{g_m}, }(s_i)\mapsto [D_{(s_i)}].
 \end{equation} 
 
 Then every fibre of $\Phi$ through a smooth spectral curve   is of dimension $g+1$. In particular, 
$\Phi$ is nowhere quasi-finite. 
\end{cor}

\begin{proof}
The arguments are identical to the ones in Sections \ref{sec:fibredim} and \ref{sec:proof}.
\end{proof}

\begin{remark}
It turns out that Bott vanishing holds for the surface $T$ with respect to spectral curves,
i.e.\ $H^1(T,\kt_T\otimes\ko(-D))=0$ for $D\in|\pi^\ast\omega_C^m\otimes\ko_\pi(m)|$.
To proves this, one uses the exact sequence
$\xymatrix@C=15pt{0\ar[r]&\kt_\pi\cong\pi^\ast\omega_C\otimes\ko_\pi(2)\ar[r]&\kt_T\ar[r]&\pi^\ast\kt_C\ar[r]&0}$ for the projective bundle $\pi\colon T=\PP(\ko_C\oplus\omega_C)\to C$.

Geometrically this means that the map $\tilde\Phi\colon(D\subset T)\mapsto D$
is infinitesimally injective. Here, $T$ can vary as well, either as a general surface, i.e.\
letting both $C$ and the bundle $\ko_C\oplus\omega_C$ vary, or only as the completion
of the cotangent bundle of deformations of $C$. So, the relative tangent spaces
of $\Phi$ and $\tilde\Phi$ at a spectral curve $D$ are
$$\kt_\Phi([D])\cong H^0(C,\omega_C)\oplus H^0(C,\ko_C)\text{ and }\kt_{\tilde\phi}([D])\cong H^1(T,\kt_T(-D))=0.$$
Note this does not contradict the idea of Lemma \ref{lem:BVqf}, as for the map $\tilde\Phi$
one divides by the action of the group ${\rm Aut}(T)$ whose tangent space is exactly $\kt_\Phi([D])$.
\end{remark}

\subsection{}\label{sec:VGVspectral}
 Specialisation from curves $D\in |mH|$ on a K3 surface to spectral curves
$D\twoheadrightarrow C$ of degree $m$ immediately yields a version of Corollary \ref{cor:dim2} and Theorem \ref{thm:main}. The bound here is worse but below we will see that
combined with a conjecture of Matsushita it would in fact suffice to prove Theorem \ref{thm:main}, thereby giving an alternative approach to our main result.

\begin{cor}\label{cor:atmostg}
Assume $|H|$ is a base point free, ample
linear system on a K3 surface $S$ of genus $g$. Then for all $m\geq 1$ the generic fibre of the rational
map $\xymatrix{\Phi\colon|mH|\ar@{-->}[r]&M_{g_m}}$ is of dimension at most $g+1$.
In other words, a generic curve $D\in |mH|$ satisfies
$$h^0(D,\kt_S|_D)\leq g+1\text{ and }\dim \Phi^{-1}([D])\leq g+1.$$
\end{cor}

\begin{proof}
The assertion is certainly true for $m=1$, so let us restrict to the case $m\geq 2$. 
Fix a smooth curve $C\in |H|$ and consider the degeneration of $C\subset S$
to its normal cone $C\subset V={\mathbb V}( \omega_C)$ or, alternatively,
to $C\subset \bar T$, where $\bar T$ is the contraction of the section at infinity  $\PP(\ko_C)\subset T=\PP(\ko_C\oplus \omega_C)$. 

In the process, the linear system $|mH|=|mC|$ on $S$ 
specialies to the linear system $|\pi^\ast\omega_C^m\otimes\ko_\pi(m)|=|mC|$ on 
$V$ (or on $\bar T$). Thus, spectral curves $D\to C$ can be seen as specialisations
of curves $D_t\in |mH|$,  $t\ne0$, on the constant family of K3 surfaces $S$.
By semi-continuity, $h^0(D,\kt_S|_D)\leq g+1$ follows from the corresponding
statement $h^0(D,\kt_V|_D)=g+1$ for spectral curves, see Lemma \ref{lem:HM}.

Alternatively and more geometrically, one can view (\ref{eqn:HM}) as the specialisation of
the rational map $\xymatrix{\Phi\colon|H|\ar@{-->}[r]&M_{g_m},}$ which also proves the claim.
\end{proof}

\subsection{}\label{sec:VGVspectral2} Corollary \ref{cor:atmostg} can be combined with a result of van Geemen--Voisin \cite{VGV} to reprove Theorem \ref{thm:main} for the generic polarized K3 surface $(S,H)$. To make
the connection, we need to recall how to associate a Lagrangian fibration
to the linear system $|mH|$ on a K3 surface. For this consider the Mukai vector
$v_m=(0,mH,s)$ with $s$ an integer prime to $m$. Then the moduli space
$M(v_m)$  of sheaves $E$ with Mukai vector $v(E)=v_m$ which are
stable with respect to a generic polarization is a projective hyperk\"ahler manifold deformation equivalent to the Hilbert scheme $S^{[g_m]}$, cf.\ \cite[Sec.\ 6.2]{HL} or \cite[Sec.\ 10.2]{HuyK3} for the precise statement and references. The map that sends $E\in M(v_m)$ to its scheme-theoretic support defines a Lagrangian fibration
$$f\colon M(v_m)\to |mH|.$$ The  fibre over a generic curve $D\in|mH|$ is isomorphic to the Jacobian
$J(D)$. The infinitesimal Torelli theorem for curves ensures that 
the family of curves $D\in |mH|$ has maximal variation if and only if the same holds
for the abelian fibres of the Lagrangian fibration $f$. More generally,
the amount of variation is the same for $M(v_m)\to |mH|$ and the universal curve ${\mathcal D}\to|mH|$.

\begin{cor}\label{cor:Yagna} For each degree $d$, there exists a Zariski dense open subset
$ U\subset N_d$ in the moduli space of polarized K3 surfaces  $(S,H)$ of degree $(H.H)=d$ such that for every $(S,H)\in U$ and $m\geq 2$ the rational map $\xymatrix{|mH|\ar@{-->}[r]& M_{g_m}}$ is generically quasi-finite.
\end{cor}

\begin{proof} Quasi-finiteness is an open condition. Therefore, it is enough to 
find one polarized K3 surface $(S,H)$ for which the assertion holds. It is
known that the very general polarized K3 surface of degree $d$ has Picard
number one and a trivial endomorphism field $\QQ\cong{\rm End}_{\rm Hdg}(T(S))$,
cf.\ \cite[Lem.\ 9]{VGV}.  In particular,
its Mumford--Tate group is maximal and, thus, the assumptions of \cite[Thm.\ 5]{VGV} are satisfied. Hence, either the curves in the linear system $|mH|$, $m\geq 2$ have maximal or trivial variation. The latter is excluded by Corollary \ref{cor:atmostg}.
\end{proof}

\begin{remark} (i) The technique actually shows more. It shows that in any irreducible
closed subset $Z\subset N_d$ of the moduli space of polarized K3 surface containing at
least one K3 surface $(S,H)\in Z$ with maximal Mumford--Tate group, the set of
K3 surfaces $(S,H)\in Z$ with maximal variation in $|mH|$, $m\geq 2$, is Zariski open and dense. 
 \smallskip
 
 (ii) Note that in the above proof, instead of Corollary \ref{cor:atmostg} one could have evoked Corollaries
\ref{cor:Hein} and \ref{cor:dim2}, which would have allowed us to prove the result
for $m=1$ and $g>2$ as well. However,  the approach via Hitchin systems is easier for it avoids any restriction result as Theorem \ref{thm:Hein}.
\end{remark}

\section{Primitive linear systems and Matsushita's conjecture}\label{sec:counterexamples}
In the proof of Corollary \ref{cor:Yagna} we used \cite[Thm.\ 5]{VGV} which proves
a conjecture of Matsushita under additional assumptions on
the rank of the transcendental lattice $T(X)\subset H^2(X,\ZZ)$ and its
 Mumford--Tate group. More precisely, van Geemen--Voisin answer affirmatively the following conjecture for very general hyperk\"ahler manifolds.

\begin{conj}[Matsushita]\label{conj:Matsu}
Let $f\colon X\to\PP^m$ be a Lagrangian fibration of a projective hyper\-k\"ahler manifold.
Then, either the family of smooth fibres, which are abelian varieties, has maximal variation or it is isotrivial.
\end{conj}

\subsection{} Our main result Theorem \ref{thm:main} can be seen as further evidence for the conjecture, because it immediately confirms it for essentially all Lagrangian fibrations
$f\colon M(v_m)\to \PP^{g_m}$ associated with the non-primitive, base point free ample 
linear system $|mH|$, $m\geq2$, see  Section \ref{sec:VGVspectral2} for the notation and explanations. As it is in principle possible that  Conjecture \ref{conj:Matsu}
holds generically but fails for special Lagrangian fibrations, the fact that
Theorem \ref{thm:main} holds for essentially all K3 surfaces is seen as
strong evidence in favour of the conjecture.

\medskip

Conversely, Matsushita's conjecture combined with Corollary \ref{cor:Hein}
would provide a quick way to prove Theorem \ref{thm:main} and not only
for $m\geq 2$, but in fact for all $m\geq 1$.

\begin{prop}\label{prop:mainMatsu}
Let $H$ be a base point free, ample linear system on a K3 surface $S$ of genus $g\geq 3$.
Assume Conjecture \ref{conj:Matsu} holds for the moduli space $M(v)$ of stable sheaves
with Mukai vector $v=(0,H,s)$ and the natural Lagrangian fibration $f\colon M(v)\to \PP^g$. Then the linear system $|H|$ has maximal variation.
\end{prop}

\begin{proof}
Due to Corollaries \ref{cor:Hein} and \ref{cor:dim2}, we know that an isotrivial family of curves through the generic point in $|H|$  is of dimension at most two. In particular, for $g\geq3$ the Lagrangian fibration $M(v)\to\PP^g$ cannot be isotrivial. Hence, by Conjecture \ref{conj:Matsu}, it has to have maximal variation.
\end{proof}

As \cite[Thm.\ 5]{VGV} verifies Conjecture \ref{conj:Matsu} under additional hypotheses,
one immediately obtains the following consequence.

\begin{cor}
Assume $S$ is a K3 surface with Picard number $\rho(S)\leq 17$ and Hodge
endomorphism field ${\rm End}(T(S)\otimes \QQ)\cong\QQ$. Then
any ample, base point free linear system $|H|$ has maximal variation.\qed
\end{cor}

\subsection{} To conclude, let us look at the case of primitive linear systems of genus two
which are not covered by Proposition \ref{prop:mainMatsu} and for which maximal variation can in fact be shown directly and without relying on Conjecture \ref{conj:Matsu}.

Recall that a base point free, ample linear system $|H|$ of genus two, i.e.\ $(H.H)=2$,
induces a degree two morphism $\pi\colon S\to \PP^2$ which is branched along a smooth
sextic curve $D_0\in |\ko_{\PP^2}(1)|$. 

\begin{prop}\label{prop:Yagna}
A base point free, ample linear system $|H|$ of degree $(H.H)=2$ on a K3 surface
has maximal variation.
\end{prop}

\begin{proof} According to \cite[Lem.\ 3.16(d)]{EV}, we have $\pi_\ast\Omega_S\cong\Omega_{\PP^2}\oplus \Omega_{\PP^2}({\rm log}D_0)(-3)$
and, therefore, for any line $\ell\subset\PP^2$
$$H^0(\pi^{-1}(\ell),\Omega_S|_{\pi^{-1}(\ell)})\cong H^0(\ell,(\Omega_{\PP^2}\oplus\Omega_{\PP^2}({\rm log}D_0)(-3))|_\ell).$$

Now, both bundles, $\Omega_{\PP^2}$ and $\Omega_{\PP^2}({\rm log}D_0)(-3)$, are 
stable of rank two with determinant $\ko(-3)$. 
Indeed, the stability of $\Omega_{\PP^2}$ is well known and for $\Omega_{\PP^2}({\rm log}D_0)$ it is deduced by a standard argument from the exact residue sequence $$\xymatrix@C=19pt{0\ar[r]&\Omega_{\PP^2}\ar[r]&\Omega_{\PP^2}({\rm log}D_0)\ar[r]&\ko_{D_0}\ar[r]&0.}$$

The Grauert--M\"ulich theorem,
cf.\ \cite[Thm.\ 3.0.1]{HL}, then proves that the restriction of both bundles,
$\Omega_{\PP^2}$ and $\Omega_{\PP^2}({\rm log}D_0)(-3)$, to the generic line $\ell\subset\PP^2$ is isomorphic to $\ko(-1)\oplus\ko(-2)$. Therefore,
$H^0(C,\kt_S|_C)\cong H^0(\pi^{-1}(\ell),\Omega_S|_{\pi^{-1}(\ell)})=0$ for the generic
curve $C=\pi^{-1}(\ell)\in |H|\cong|\ko(1)|$.
\end{proof}

\end{document}